\documentclass[12pt,reqno]{amsart}

\newtheorem{theorem}{Theorem}[section]

\usepackage{amsmath,amssymb,amsfonts,}
\usepackage[dvips]{graphics}
\usepackage{amssymb}
\usepackage{amsfonts}
\usepackage{latexsym}
\usepackage{cite}
\usepackage{mathrsfs}

\theoremstyle{definition}

\theoremstyle{remark}

\theoremstyle{observation}

\theoremstyle{conjecture}

\theoremstyle{corollary}

\theoremstyle{plain}
\newtheorem{claim}{Claim}

\theoremstyle{proposition}

\theoremstyle{problem}

\numberwithin{equation}{section}

\large\normalsize
\RequirePackage[normalem]{ulem} 
\RequirePackage{color}\definecolor{RED}{rgb}{1,0,0}\definecolor{BLUE}{rgb}{0,0,1} 
\providecommand{\DIFaddbegin}{} 
\providecommand{\DIFaddend}{} 
\providecommand{\DIFdelbegin}{} 
\providecommand{\DIFdelend}{} 
\providecommand{\DIFaddbeginFL}{} 
\providecommand{\DIFaddendFL}{} 
\providecommand{\DIFdelbeginFL}{} 
\providecommand{\DIFdelendFL}{} 
\newcommand{\DIFscaledelfig}{0.5}
\RequirePackage{settobox} 
\RequirePackage{letltxmacro} 
\newsavebox{\DIFdelgraphicsbox} 
\newlength{\DIFdelgraphicswidth} 
\newlength{\DIFdelgraphicsheight} 
\LetLtxMacro{\DIFOincludegraphics}{\includegraphics} 
\newcommand{\DIFaddincludegraphics}[2][]{{\color{blue}\fbox{\DIFOincludegraphics[#1]{#2}}}} 
\newcommand{\DIFdelincludegraphics}[2][]{
\sbox{\DIFdelgraphicsbox}{\DIFOincludegraphics[#1]{#2}}
\settoboxwidth{\DIFdelgraphicswidth}{\DIFdelgraphicsbox} 
\settoboxtotalheight{\DIFdelgraphicsheight}{\DIFdelgraphicsbox} 
\scalebox{\DIFscaledelfig}{
\parbox[b]{\DIFdelgraphicswidth}{\usebox{\DIFdelgraphicsbox}\\[-\baselineskip] \rule{\DIFdelgraphicswidth}{0em}}\llap{\resizebox{\DIFdelgraphicswidth}{\DIFdelgraphicsheight}{
\setlength{\unitlength}{\DIFdelgraphicswidth}
\begin{picture}(1,1)
\thicklines\linethickness{2pt} 
{\color[rgb]{1,0,0}\put(0,0){\framebox(1,1){}}}
{\color[rgb]{1,0,0}\put(0,0){\line( 1,1){1}}}
{\color[rgb]{1,0,0}\put(0,1){\line(1,-1){1}}}
\end{picture}
}\hspace*{3pt}}} 
} 
\LetLtxMacro{\DIFOaddbegin}{\DIFaddbegin} 
\LetLtxMacro{\DIFOaddend}{\DIFaddend} 
\LetLtxMacro{\DIFOdelbegin}{\DIFdelbegin} 
\LetLtxMacro{\DIFOdelend}{\DIFdelend} 
\DeclareRobustCommand{\DIFaddbegin}{\DIFOaddbegin \let\includegraphics\DIFaddincludegraphics} 
\DeclareRobustCommand{\DIFaddend}{\DIFOaddend \let\includegraphics\DIFOincludegraphics} 
\DeclareRobustCommand{\DIFdelbegin}{\DIFOdelbegin \let\includegraphics\DIFdelincludegraphics} 
\DeclareRobustCommand{\DIFdelend}{\DIFOaddend \let\includegraphics\DIFOincludegraphics} 
\LetLtxMacro{\DIFOaddbeginFL}{\DIFaddbeginFL} 
\LetLtxMacro{\DIFOaddendFL}{\DIFaddendFL} 
\LetLtxMacro{\DIFOdelbeginFL}{\DIFdelbeginFL} 
\LetLtxMacro{\DIFOdelendFL}{\DIFdelendFL} 
\DeclareRobustCommand{\DIFaddbeginFL}{\DIFOaddbeginFL \let\includegraphics\DIFaddincludegraphics} 
\DeclareRobustCommand{\DIFaddendFL}{\DIFOaddendFL \let\includegraphics\DIFOincludegraphics} 
\DeclareRobustCommand{\DIFdelbeginFL}{\DIFOdelbeginFL \let\includegraphics\DIFdelincludegraphics} 
\DeclareRobustCommand{\DIFdelendFL}{\DIFOaddendFL \let\includegraphics\DIFOincludegraphics} 

\begin{document}
\begin{sloppypar}
\begin{center}

{\Large\bf An Ore-type condition for regular factors}
\\[20pt]
{Jingchao\ Lai, Weigen\ Yan*}\footnote{*Corresponding Author.\newline\hspace*{5mm}{Supported by NSFC Grant (Nos. 12571366, 12071180) .} \newline\hspace*{5mm}{Email address: jclai1999@163.com (J. Lai), weigenyan@263.net (W. Yan).} }
\\[10pt]
\footnotesize { 
School of Science, Jimei University, Xiamen 361021, China}\\
\end{center}

\title{}

\begin{abstract}
Let $G$ be a simple graph of order $n$ satisfying the following Ore-type condition: For any two nonadjacent vertices $x$ and $y$ of $G$, $d_G(x)+d_G(y)\geq n+k-2$, where $1\leq k\leq n-1$, $kn$ is even and $d_G(x)$ is the degree of $x$ in $G$. It is well known that $G$ has a $k$-factor for $k=1$ or $2$.
Lu and Ning (J. Graph Theory, 94(2020), 307-319) proved that if $k\geq n/2$, then $G$ has a $k$-factor. In this paper, we show that $G$ has a $k$-factor for any $1\leq k\leq n-1$.

{\sl Keywords:}\quad Regular factor; Ore-type condition; Tutte's theorem.\\
{\sl MSC(2020):}\quad 05C07; 05C70
\end{abstract}

\maketitle
\section{Introduction}

The graphs considered in this paper are simple if not specified. Let $G$ be a graph. Denote by $V(G)$, $E(G)$ and $\delta(G)$ the vertex set, the edge set and the minimum degree of $G$, respectively. Denote by $|V(G)|$ and $|E(G)|$ the {\em order} and {\em size} of $G$, respectively.
For a vertex $v\in V(G)$ and  a subgraph $H$ of $G$, denote by $N_G(v)$ the set of neighbors of $v$ in $G$, and $d_G(v)=|N_G(v)|$, and define $N_H(v)=N_G(v)\cap V(H)$ and $d_H(v)=|N_H(v)|$.
For two disjoint subsets $S,T$ of $V(G)$, let $E_G(S,T)$ denote the set of edges between $S$ and $T$ in $G$, and set $e_G(S,T)=|E_G(S,T)|$. If $S=\{v\}$, we write $E_G(v,T)$ and $e_G(v,T)$ in place of $E_G(\{v\},T)$ and $e_G(\{v\},T)$, respectively. For a positive integer $n$, set $[n]=\{1,2,\ldots,n\}$.
Other terminology and notation not defined here can be found in \cite{BM08}.

The existence of prescribed subgraphs is a central and long-standing theme in graph theory. Perhaps the best-known result is Dirac's theorem \cite{Dirac1952} in 1952, which states that every graph of order $n$ contains a Hamilton cycle if $\delta(G)\ge n/2$. Ore \cite{Ore1960} later generalized Dirac's theorem by imposing a degree sum lower bound on all pairs of nonadjacent vertices in a graph.

Let $k\geq 1$ be an integer. A {\em $k$-factor} of a graph $G$ is a spanning $k$-regular subgraph of $G$. The seminal work of Tutte in 1952 established a complete necessary and sufficient condition for the existence of $k$-factors, laying the theoretical foundation for the entire field. However, Tutte's criterion requires verifying a combinatorial inequality over all pairs of disjoint vertex subsets, which is computationally prohibitive for large graphs. This inherent limitation has motivated decades of research devoted to deriving simple, easily verifiable sufficient conditions for the existence of $k$-factors. Among all categories of sufficient conditions, degree-based criteria constitute the most extensively studied and fruitful class. Following the classical Dirac-type minimum degree paradigm, the Ore-type framework - which imposes a lower bound on the degree sum of every pair of nonadjacent vertices - also serves as a highly flexible and powerful research tool.

For a graph $G$, the invariant $\sigma_2(G)$ is defined to be the minimum degree sum of two nonadjacent vertices of $G$, i.e.,
$$\sigma_2(G)=\mathrm{min}\{d_G(x)+d_G(y):x,y\in V(G),x\neq y,xy\notin E(G)\}$$
if $G$ is not complete; otherwise, let $\sigma_2(G)=+\infty$.

A classical result of Ore \cite{Ore1960} says that a simple graph $G$ of order $n$ has a Hamilton path (hence $G$ has a 1-factor when $n$ is even) if $\sigma_2(G)\geq n-1$ and $G$ has a Hamilton cycle (i.e., a $2$-factor) if $\sigma_2(G)\geq n$. {\bf A natural problem is: Suppose that a simple graph $G$ of order $n$ satisfies the following Ore-type condition: $\sigma_2(G)\geq n+k-2$, where $1\leq k\leq n-1$ and $kn$ is even. Has $G$ a $k$-factor?}

Several progressively refined degree-based results concerning the existence of $k$-factors are well-known, as stated below.
\begin{theorem}\label{Nishimura}
Let $k \ge 2$ be an integer, and $G$ be a connected graph of order $n$ such that $\delta(G) \ge k$ and $kn$ is even. If one of the following three conditions holds, then $G$ has a $k$-factor.\\
{\em(1).} $\delta(G) \ge n/2$ and $n \ge 4k-5$. {\em (Egawa and Enomoto~\cite{EE1989}; Katerinis~\cite{K1985})}\\
{\em(2).} $\sigma_2(G)\geq n$ and $n \ge 4k-5$.  {\em (Iida and Nishimura~\cite{IN1991})}\\
{\em(3).} $\max\{d_G(x), d_G(y)\} \ge n/2$ for every pair of nonadjacent vertices $x$ and $y$ of $G$ and $n \ge 4k-3$. {\em (Nishimura~\cite{N1992})}.
\end{theorem}

By the theorem above, if $G$ satisfies the Ore-type condition above, then $G$ has a $k$-factor if $k\leq (n+5)/4$. On the other hand, Lu and Ning \cite{LN2019} obtained the following result for the existence of large $k$-factors.
\begin{theorem}[\upshape\cite{LN2019}]\label{LN2019}
Let $n$ and $k$ be two integers such that $1\leq k \leq n-1$ and $kn$ is even. Let $G$ be a graph of order $n$.  If $\sigma_2(G)\geq n+k-2$ and $k\geq n/2$, then $G$ contains a $k$-factor.
\end{theorem}

Hence Theorems \ref{Nishimura} and \ref{LN2019} partially answer the problem above. In this paper, we obtain the following result which completely solves this problem.

\begin{theorem}\label{main}
Let $n$ and $k$ be two integers such that $1\leq k \leq n-1$ and $kn$ is even. Let $G$ be a graph of order $n$.  If $\sigma_2(G)\geq n+k-2$, then $G$ contains a $k$-factor.
\end{theorem}

The assumption that $\sigma_2(G)\geq n+k-2$ in Theorem \ref{main} cannot be weakened any further. Let $G=K_{k-1}+(K_{n-k}\cup K_1)$, where the plus sign denotes join. Obviously, $G$ contains no $k$-factor. However, it is easy to check that $\sigma_2(G)= n+k-3$.

The main tools in the work of Lu and Ning \cite{LN2019} are Tutte's $k$-factor theorem and the Karush-Kuhn-Tucker theorem. Our main technique is to optimize and simplify the method of Lu and Ning. The rest of this paper is organized as follows. In Section 2, we introduce necessary preliminaries. In Section 3, we complete the proof of Theorem \ref{main}.

\section{Preliminaries}
In this section, we first introduce some notation and terminology related to the Tutte's $k$-factor theorem. For a graph $G$ and any pair of disjoint subsets $S,T\subset V(G)$, a component $C$ of $G-S-T$ is called a \textit{$k$-odd-component} if
\[
e_G(V(C),T)+k\,|V(C)|\equiv 1 \pmod{2}.
\]
Let $q(S,T)$ denote the number of $k$-odd components of $G-S-T$. Unless otherwise specified, we always let $C_1,\ldots,C_q$ be all $k$-odd-components of $G-S-T$, and let $U=\bigcup_{i=1}^qC_i$, where $q=q(S,T)$.

The well-known Tutte's $k$-factor theorem is stated as follows.

\begin{theorem}[Tutte \cite{Tutte1952}]\label{Tutte1952}
Let $k$ be a positive integer. A graph $G$ contains no $k$-factor if and only if there exist disjoint subsets $S,T\subset V(G)$, such that
\begin{equation}\label{E1}
k\,|S|-k\,|T|+\sum_{x\in T}d_{G-S}(x)-q(S,T)\le -2.
\end{equation}
\end{theorem}

From Tutte's theorem, Katerinis and Woodall \cite{KW1987} deduced the following.

\begin{theorem}[Katerinis and Woodall \cite{KW1987}]\label{KW1987}
Let $k\ge 1$ be an integer. If a graph $G$ contains no $k$-factor, then there exist two disjoint subsets $S,T\subset V(G)$ such that there holds {\em (\ref{E1})}, and for all $v\in V(U)$
\begin{equation}\label{E2}
e_G(v,T)\le k-1,\quad \text{and}
\end{equation}
\begin{equation}\label{E3}
d_{G-S}(v)\ge k+1,\quad \text{and}
\end{equation}
\begin{equation}\label{E4}
|V(C_i)|\geq 3 \quad \text{for every $i\in [q]$.}
\end{equation}
\end{theorem}

\section{Proof of Theorem \ref{main}}

{\bf{Proof of Theorem \ref{main}}}
To prove the theorem, we argue by contradiction. Suppose that $G$ contains no $k$-factor. By Theorem \ref{KW1987}, we can choose disjoint $S,T\subset V(G)$ satisfying Ineqs. (\ref{E1})-(\ref{E4}). Let $C_1,\dots,C_q$ be all $k$-odd components of $G-S-T$, and let $U=\bigcup_{i=1}^qC_i$. Let $s=|S|$, $t=|T|$ and $c_i=|V(C_i)|$ for $i\in [q]$. Hence, by Ineqs. (\ref{E2})-(\ref{E4}), $e_G(v,T)\le k-1$  and $d_{G-S}(v)\ge k+1$ for every $v\in V(U)$, and $c_i\geq 3$ for every $i\in[q]$. So, $|V(U)|=\sum_{i=1}^qc_i\geq 3q$ and $n\geq s+t+3q$.

If $T\neq \emptyset$, then let $h_1:=\min\{d_{G-S}(x):x\in T\}$, and let $u_1\in T$ such that $d_{G-S}(u_1)=h_1$. Set $P=(N_G(u_1)\cap T)\cup\{u_1\}$ and $p=|P|$. If $T\backslash P \neq\emptyset$, then let $h_2:=\min\{d_{G-S}(x):x\in T\backslash P\}$ and choose $u_2\in T\backslash P$ such that $d_{G-S}(u_2)=h_2$. Obviously, $p=d_T(u_1)+1$, $u_1u_2\notin E(G)$ and $h_1\leq h_2$.

\begin{claim}\label{k-connected}
$G$ is $k$-connected; hence,  $\delta(G)\ge k$.
\end{claim}

\begin{proof}
If there exists a vertex subset $W$ such that $G-W$ is disconnected, then let $C_1',C_2'$ be two components of $G-W$. Let $x\in V(C_1')$ and $y\in V(C_2')$. Then we can see that $xy\notin E(G)$. Hence,
\[
n + k - 2 \le d_G(x) + d_G(y) \le |V(C_1')|-1+|W|+|V(C_2')| - 1 + |W|.
\]
Note that $n \ge |V(C_1')| + |V(C_2')| + |W|$. Hence, $|W| \ge k$, and moreover, $\delta(G) \ge k$.
\end{proof}

\begin{claim}\label{h1}
 We have $q<ks+2$; hence, $T\neq \emptyset$ and $h_1\leq k-1$.
\end{claim}
\begin{proof}
If not, then $q\geq ks+2\geq 2$, which implies $G-S-T$ is not connected. That is, $S\cup T$ is a vertex-cut of $G$. By Claim \ref{k-connected}, we have $s+t\geq k$. Note that
\begin{equation}\label{n-k-0}
\sum_{i=1}^qc_i=|V(U)|\leq n-s-t\leq n-k.
\end{equation}
Recall that $|V(U)|\ge 3q$.  Then we have
\begin{equation}\label{n-k-1}
n-k\geq  3q\geq 3ks+6.
\end{equation}
For every $i\in[q]$, let $x_i\in V(C_i)$. Then $N_G(x_i)\subseteq (V(C_i)\setminus \{x_i\})\cup S\cup T$. By Ineq. (\ref{E2}),
\begin{equation*}
d_G(x_i)\leq c_i-1+s+e_G(x_i,T)\leq c_i+s+k-2.
\end{equation*}
Hence, for any $i,j\in[q]$ with $i\neq j$, we first have $x_ix_j\notin E(G)$ and then
\begin{equation*}
c_i+s+k-2+c_j+s+k-2\geq d_G(x_i)+d_G(x_j) \geq n+k-2.
\end{equation*}
That is, $c_i+c_j\geq n-k-2s+2$. Without loss of generality, we suppose that $c_1\leq c_2\leq\ldots\leq c_q$. It is clear that
\begin{equation*}
\frac{c_1+c_2}{2}\leq \frac{1}{q}\sum_{i=1}^qc_i\leq \frac{n-k}{q}.
\end{equation*}
where the second inequality holds by Ineq. (\ref{n-k-0}).  Hence we have
\begin{equation}
n-k-2s+2\leq c_1+c_2\leq \frac{2}{q}(n-k).
\end{equation}
That is, $q(n-k-2s+2)\leq 2(n-k)$. Note that by Ineq. (\ref{n-k-1}),
\begin{equation}
n-k-2s+2\ge 3ks+6-2s+2=(3k-2)s+8>0.
\end{equation}
However,
\begin{align*}
q(n-k-2s+2)-2(n-k)
&\ge (ks+2)(n-k-2s+2)-2(n-k) \tag{by $q\geq ks+2$}\\
&= ks(n-k)-(ks+2)(2s-2) \\
&\ge ks(3ks+6)-(ks+2)(2s-2) \tag{by (\ref{n-k-1})} \\
&= (ks+2)(3ks-2s+2)  \\
&\ge (ks+2)(s+2) \tag{by $k\ge 1$}\\
&>0,
\end{align*}
a contradiction. Thus $q<ks+2$. It is obvious that $T\neq \emptyset$. Otherwise, by Ineq. (\ref{E1}),
\begin{equation*}
ks-q\leq -2,
\end{equation*}
a contradiction. Recalling the definition of $h_1$, it is clear that  $d_{G-S}(x)\geq h_1$ for every $x\in T$.
If $h_1\geq k$, then by Ineq. (\ref{E1}),
\begin{align*}
q\geq ks-kt+\sum_{x\in T}d_{G-S}(x)+2\geq ks-kt+kt+2=ks+2,
\end{align*}
a contradiction. Hence, we have $h_1\leq k-1$.

The following two claims are obvious:

\begin{claim}\label{k,h1+s}
$k \le \delta(G) \le d_G(u_1) \le d_{G-S}(u_1)+s = h_1+s$.
\end{claim}
\begin{claim}\label{p,k}
$p= d_T(u_1)+1 \le d_{G-S}(u_1)+1 = h_1+1\leq k$.
\end{claim}

Finally, we require the following important result:

\begin{claim}\label{s-lower}
If $V(U)\not\subseteq N_G(u_1)$, then
\begin{equation*}
s\geq k-h_1+3(q-1).
\end{equation*}
\end{claim}
\begin{proof}
Without loss of generality, take $x\in V(C_j)\subseteq V(U)$ for some $j\in [q]$, and suppose that $xu_1\notin E(G)$. Recall that $c_i\geq 3$ for every $i\in [q]$. Hence we have
\begin{equation}\label{cj}
c_j\leq n-s-t-\sum_{i\neq j}c_i\leq n-s-t-3(q-1).
\end{equation}
It is obvious that $N_G(x)\subseteq (V(C_j)\setminus\{x\})\cup S\cup (T\setminus \{u_1\})$. Furthermore, by Ineq. (\ref{cj}),
\begin{equation}
d_G(x)\leq c_j+s+t-2\leq n-s-t-3(q-1)+s+t-2=n-3(q-1)-2.
\end{equation}
By Claim \ref{k,h1+s}, $d_G(u_1)\leq h_1+s$. So we have
\begin{equation*}
n+k-2\leq d_G(x)+d_G(u_1)\leq n-3(q-1)-2+h_1+s.
\end{equation*}
That is,
\begin{equation*}
s\geq k-h_1+3(q-1).
\end{equation*}
The result thus holds.
\end{proof}

We now consider three cases to derive a contradiction.

\textbf{Case 1.} $T\setminus P= \emptyset$. That is, $T=P$.

Let $r=k-h_1$. By Claims \ref{h1}-\ref{p,k}, we have $r\geq 1$, $s\geq r$ and $t=p\leq k$. Furthermore, by Ineq. ({\ref{E1}}), we have
\begin{align}
q&\ge ks-kt+\sum_{x\in T}d_{G-S}(x)+2 \notag\\
&\ge ks-kt+h_1t+2 \notag \\
&=ks-rt+2 \label{ks-rt+2}\\
&\ge2. \tag{by $k\geq t$,\;$s\ge r$}
\end{align}
\begin{claim}\label{case1-claim}
 $V(U)\not\subseteq N_G(u_1)$.
\end{claim}
\begin{proof}
If $V(U)\subseteq N_G(u_1)$, then $h_1=d_{G-S}(u_1)\geq |V(U)|+t-1$, since $T=P$. Note that for every $x\in V(U)$,
$N_{G-S}(x)\subseteq (V(U)\setminus \{x\})\cup T$. Hence,
\begin{equation*}
d_{G-S}(x)\leq |V(U)|-1+t\leq h_1.
\end{equation*}
But by Ineq. (\ref{E3}), $d_{G-S}(x)\geq k+1$. That is, $h_1\geq k+1$, which contradicts Claim \ref{h1}.
\end{proof}
By Claims \ref{case1-claim} and \ref{s-lower}, we have $s\geq r+3(q-1)$. Furthermore, by Ineq. (\ref{ks-rt+2}),
\begin{align*}
q &\ge ks-rt+2 \\
  &\ge k\big(r+3(q-1)\big)-rk+2 \tag{by $s\ge r+3(q-1),\, t\le k$}\\
  &= 3k(q-1)+2\\
  &\geq 3(q-1)+2 \tag{by $k\geq 1$}\\
  &>q \tag{by $q\geq 2$},
\end{align*}
a contradiction. This completes the proof for Case 1.

\textbf{Case 2.} $T\setminus P\neq \emptyset$ and $h_2\geq k$.

Let $r=k-h_1$. By Claims \ref{h1}-\ref{p,k}, we have $r\geq 1$, $s\geq r$ and $p\leq h_1+1\leq k$. Recall that $h_2:=\min\{d_{G-S}(x):x\in T\backslash P\}$ if $T\setminus P\neq \emptyset$. By Ineq. ({\ref{E1}}), we have
\begin{align}\label{case2-q}
q &\ge ks-kt+\sum_{x\in P}d_{G-S}(x)+\sum_{x\in T\setminus P}d_{G-S}(x)+2 \notag \\
  &\ge ks-kt+h_1p+h_2(t-p)+2 \notag\\
  &\ge ks-kt+h_1p+k(t-p)+2 \tag{by $h_2\ge k$} \notag\\
  &= ks-rp+2 \label{ks-rp+2}\\
  &\ge r(k-p)+2 \label{r(k-p)+2}\\
  &\ge 2 \notag.
\end{align}
\begin{claim}\label{case-claim}
$V(U)\subseteq N_G(u_1)$.
\end{claim}
\begin{proof}
If $V(U)\not\subseteq N_G(u_1)$, then $s\geq r+3(q-1)$ by Claim \ref{s-lower}.  Furthermore, by Ineq. (\ref{ks-rp+2}),
\begin{align*}
q &\ge ks - rp + 2 \\
  &\ge k\big(r+3(q-1)\big)-rk+2 \tag{by $s\ge r+3(q-1),\, p\le k$}\\
  &= 3k(q-1)+2 \\
  &\ge 3(q-1)+2 \tag{by $k\ge 1$}\\
  &> q \tag{by $q\ge 2$},
\end{align*}
a contradiction. The claim thus holds.
\end{proof}
Since $V(U)\subseteq N_G(u_1)$ and $|V(U)|\ge 3q$, we have $h_1=d_{G-S}(u_1)\geq 3q+p-1$. That is,
\begin{equation}\label{case2-p}
p\leq h_1-3q+1.
\end{equation}
Furthermore, by Ineqs. (\ref{r(k-p)+2}) and ({\ref{case2-p}}), we have
\begin{equation*}
q\ge r(k-p)+2\ge r(k-h_1+3q-1)+2\geq (1+3q-1)+2>q,
\end{equation*}
where the third inequality holds since $r=k-h_1\geq1$, a contradiction. This completes the proof for Case 2.

\textbf{Case 3.} $T\setminus P\neq \emptyset$ and $h_2\leq k-1$.

Recall that $u_1u_2\notin E(G)$, so we have
\begin{equation*}
n+k-2\leq d_G(u_1)+d_G(u_2)\leq d_{G-S}(u_1)+s+d_{G-S}(u_2)+s=h_1+h_2+2s.
\end{equation*}
That is,
\begin{equation}\label{s-1}
s\geq \frac{n+k-2-h_1-h_2}{2}.
\end{equation}
On the other hand, by Claim \ref{k,h1+s},
\begin{equation}\label{s-2}
s\geq k-h_1.
\end{equation}
Let $\lambda=\frac{h_2}{2k-h_2}$. Then $0\leq \lambda <1$ since $h_2\leq k-1$, and
\begin{equation}\label{ww}
(2k-h_2)\lambda=h_2\;\;\mathrm{and}\;\; \frac{(2k-h_2)(1-\lambda)}{2}=k-h_2.
\end{equation}
By Ineqs. (\ref{s-1}) and (\ref{s-2}),
\begin{equation}\label{s-3}
s=\lambda s+(1-\lambda)s\geq \lambda (k-h_1)+(1-\lambda)\frac{n+k-2-h_1-h_2}{2}.
\end{equation}
Since $2k-h_2>0$, by Ineqs. (\ref{ww}) and (\ref{s-3}),
\begin{align}\label{s-4}
(2k-h_2)s&\ge (2k-h_2)\lambda (k-h_1)+\frac{(2k-h_2)(1-\lambda)}{2}(n+k-2-h_1-h_2) \notag\\
&=h_2(k-h_1)+(k-h_2)(n+k-2-h_1-h_2).
\end{align}
Note that $n\geq s+t+|V(U)|\geq s+t+3q$. That is,
\begin{equation}\label{case3-t}
t\leq n-s-3q.
\end{equation}
Obviously, $h_1-h_2\leq 0$. By Claim \ref{p,k}, $p\leq h_1+1$. Hence,
\begin{equation}\label{case3-h12}
(h_1-h_2)p\ge(h_1-h_2)(h_1+1).
\end{equation}
Furthermore, by Ineq. (\ref{E1}), we have
\begin{align*}
0 &\ge ks-kt+\sum_{x\in P}d_{G-S}(x)+\sum_{x\in T\setminus P}d_{G-S}(x)+2-q \\
  &\ge ks-kt+h_1p+h_2(t-p)+2-q \\
  &= ks-(k-h_2)t+(h_1-h_2)p+2-q \\
  &\ge ks-(k-h_2)(n-s-3q)+(h_1-h_2)p+2-q \tag{by (\ref{case3-t})}\\
  &\ge ks-(k-h_2)(n-s-3q)+(h_1-h_2)(h_1+1)+2-q \tag{by (\ref{case3-h12})} \\
  &= (2k-h_2)s-(k-h_2)n+\big(3(k-h_2)-1\big)q+(h_1-h_2)(h_1+1)+2 \\
  &\ge (2k-h_2)s-(k-h_2)n+(h_1-h_2)(h_1+1)+2.
\end{align*}
\end{proof}
Let $f=(2k-h_2)s-(k-h_2)n+(h_1-h_2)(h_1+1)+2$. Then $f\leq 0$. However,
\begin{align*}
f &= (2k-h_2)s-(k-h_2)n+(h_1-h_2)(h_1+1)+2 \\
  &\ge h_2(k-h_1)+(k-h_2)(n+k-2-h_1-h_2)-(k-h_2)n+(h_1-h_2)(h_1+1)+2\tag{by (\ref{s-4})}\\
  &= h_2(k-h_1)+(k-h_2)(k-2-h_1-h_2)+(h_1-h_2)(h_1+1)+2.
\end{align*}
Set $a=k-1-h_1$ and $b=k-1-h_2$. Then $a\geq b\geq0$. By elementary calculations,
\begin{align*}
 f&\ge h_2(k-h_1)+(k-h_2)(k-2-h_1-h_2)+(h_1-h_2)(h_1+1)+2\\
 &=1+a^2-ab+b^2\\
 &=1+(a-b)^2+ab\\
 &\geq 1,
\end{align*}
which contradicts $f\leq 0$. This completes the proof for Case 3, and hence the theorem follows.

\vspace{3mm}
\noindent\textbf{Declaration of competing interest}

There is no conflict of interest.

\noindent\textbf{Data availability}

No data was used for the research described in the article.

\end{sloppypar}

\begin{thebibliography}{99}

\bibitem{BM08}
J. A. Bondy, U. S. R. Murty, Graph theory, Graduate Texts in Mathematics, Springer,
New York, 2008.

\bibitem{Dirac1952}
G. A. Dirac, \textit{Some theorems on abstract graphs}, Proc. London. Math. Soc., 2 (1952), 69-81.

\bibitem{EE1989}
Y. Egawa, H. Enomoto, \textit{Sufficient conditions for the existence of $k$-factors}, Recent Studies in Graph Theory, V. R. Kulli, Ed., Vishwa International Publication, India (1989) 96-105.

\bibitem{IN1991}
T. Iida, T. Nishimura, \textit{An Ore-type condition for the existence of $k$-factors in graphs}, Graphs and Combinat., 7 (1991), 353-361.

\bibitem{K1985}
P. Katerinis, \textit{Minimum degree of a graph and the existence of $k$-factors}, Proc. Indian Acad. Sci (Math. Sci.), 94 (1985), 123-127.

\bibitem{KW1987}
P. Katerinis, D. R. Woodall, \textit{Binding numbers of graphs and the existence of $k$-factors}, Quart. J. Math., 38 (1987), 221-228.

\bibitem{LN2019}
H. L. Lu, B. Ning, \textit{An Ore-type condition for large $k$-factor and disjoint perfect matchings}, J. Graph Theory, 94 (2020), 307-319.

\bibitem{N1992}
T. Nishimura, \textit{A degree condition for the existence of $k$-factors}, J. Graph Theory, 16 (1992), 141-151.

\bibitem{Ore1960}
O. Ore, \textit{Note on Hamilton circuits}, Am. Math. Mon., 67 (1960), 55.

\bibitem{Tutte1952}
W. T. Tutte, \textit{The factors of graphs}, Canad. J. Math., 4 (1952), 314-328.

\end{thebibliography}
\end{document}